\documentclass{amsart}

\usepackage{amsfonts}
\usepackage{amssymb}
\usepackage{amsmath}

\usepackage[arrow,matrix,curve]{xy}

\newtheorem{theorem}{Theorem}[section]
\newtheorem{definition}[theorem]{Definition}
\newtheorem{lemma}[theorem]{Lemma} 
\newtheorem{proposition}[theorem]{Proposition} 
\newtheorem{corollary}[theorem]{Corollary}

\newcommand{\cala}[0]{{\mathcal A}}

\newcommand{\call}[0]{{\mathcal L}}

\newcommand{\cale}[0]{{\mathcal E}}
\newcommand{\calf}[0]{{\mathcal F}}
\newcommand{\calk}[0]{{\mathcal K}}

\newcommand{\R}[0]{{\mathbb R}}
\newcommand{\Z}[0]{{\mathbb Z}}

\newcommand{\C}[0]{{\mathbb C}}

\newcommand{\ad}[0]{{\mbox{\rm Ad}}}    
\newcommand{\aut}[0]{{\mbox{ \rm Aut}}}

\newcommand{\na}[0]{\nabla}
\newcommand{\re}[1]{{(\ref{#1})}} 

\newcommand{\be}[1]{\begin{equation} \label{#1}}
\newcommand{\en}[0]{\end{equation}}
\newcommand{\id}[0]{{\rm id}}			

\newcommand{\diag}[0]{{\mbox{diag}}}
\newcommand{\entd}[0]{{\mbox{End}}}

\title{On the $K$-theory in 
splitexact 
algebraic $KK$-theory}  

\author[B. Burgstaller]{Bernhard Burgstaller}
\email{bernhardburgstaller@yahoo.de} 


\date{\today}

\begin{document}


\begin{abstract}

It is 
proven that in the 
universal 
splitexact equivariant algebraic 
$KK$-theory 
for algebras,  
the $K$-theory groups 
coincide with classical 
$K$-theory  
in the sense of 
Phillips. 
  This partially answers a question raised by 
  Kasparov. 
\end{abstract}


\maketitle



\section{Introduction} 

In \cite{kasparov1981}, Kasparov has introduced $KK$-theory for the class 
of $C^*$-algebras, which 
incorporates and enriches $K$-theory significantly
by the existence of a product between the underlying $KK$-groups. 
The product allows one to view $KK$-theory as a category 
with object class the class of $C^*$-algebras, morphism 
sets the $KK$-groups and composition 
the Kasparov product, 
and 
actually, Cuntz \cite{cuntz1983, cuntz1984, cuntz1987} and Higson \cite{higson} have found out 
how to equivalently 
interpret  
the quite 
analytically defined $KK$-theory by Kasparov 
quite algebraically as the universal category that is deduced
from the category of $C^*$-algebras such that it becomes 
{\em splitexact}, homotopy invariant and stable. 
This may be concretely done by adding certain morphisms to 
and imposing relations on the category of $C^*$-algebras,
see 
\cite{bgenerators}, to 
obtain $KK$-theory, or more generally also 
equivariantly to get Kasparov's $KK^G$-theory  \cite{kasparov1988}.
 
In \cite{higson2}, Higson, and in a more tractable picture
Connes and Higson in \cite{conneshigson, 
conneshigson2}, have defined 
the 
universal {\em half-exact}, homotopy invariant 
and stable category for $C^*$-algebras coined $E$-theory. 
In another direction, 
later, 
Cuntz has defined and discussed $kk$-theory \cite{cuntz}, 
which is the universal half-exact, 
diffeotopyinvariant 
and stable category for the class of 
locally convex algebras 
by a completely different concept than $E$-theory. 
 Many authors have generalized Kasparov's $KK$-theory 
 in various directions, either by enlarging the object class
 to classes of algebras or rings, or by modifying 
 the universal properties, or something else.  
   An incomplete list of further samples 
include  Lafforgue's Banach $KK$-theory 
for Banach algebras \cite{lafforgue}, 
Cuntz \cite{cuntz2}, Cuntz and Thom 
\cite{cuntzthom}, {Corti\~nas} 
and Thom \cite{cortinasthom}, 
Ellis 
\cite{ellis}, 
Garkusha \cite{garkusha1,garkusha2}, 
Grensing \cite{grensing},  
and Weidner 
\cite{weidner,weidner2}.

To get a close analogy of $KK^G$-theory for $C^*$-algebras
for algebras and rings,
in the preprint \cite{gk}  
the above mentioned splitexact universal definition of $KK^G$-theory 
was used 
for the consideration of not only the category of $C^*$-algebras   
\cite{aspects}, 
but also the category of 
rings and algebras.   
A significant point was that we allowed for quite general
{\em invertible} corner embeddings 
$A \rightarrow \calk_A(\cale \oplus A)$ which not only incorporate 
projective modules but essentially all modules $\cale$. 
G. G. Kasparov in his function as a journal 
editor asked us what the typical $K$-theory group 
$GK^G(\R,\R)$ 
(or $GK^G(\Z,\Z)$ for rings) is. 
In this note we give a partial answer to his question in
the sense that we answer it completely if we allow only 
the most  
 {\em ordinary}  invertible corner embeddings 
$(A,\alpha)  \rightarrow (M_n \otimes A, \gamma \otimes \alpha)$ in $GK^G$-theory 
(then called {\em very special} $GK^G$-theory), that is,
dropping the 
generalized,  `unprojective' invertible corner embeddings 
completely. 
The answer given in Theorem 
\ref{thm42}  
is that $GK^G(\C,A)$
is ordinary $K$-theory 
as defined in Philips 
for Fréchet algebras   \cite{phillips}   
in case the group $G$ is trivial  
(modulo obvious differences arising from the considered
class of algebras),  
which is also what comes out for
$kk(\C,A)$ in 
Cuntz \cite{cuntz} for locally convex algebras.  
Cuntz' involved proof using half-exactness
and Toeplitz extensions, 
however, completely diverges from the 
deduction in this note. 

As pointed out in Lemma 
\ref{lemma31}, very special $KK^G$-theory is sufficient 
to have a Green-Julg theorem, 
and not more than the most ordinary, very special corner 
embeddings 
are 
axiomatically declared to be invertible in Ellis' work \cite{ellis}  on the Green-Julg theorem 
for the half-exact $G$-equivariant  
generalization $kk^G$-theory of the 
non-equivariant 
half-exact 
$kk$-theory for rings by 
{Corti\~nas} and 
Thom \cite{cortinasthom}, 
itself 
a generalization or 
adaption of Cuntz' $kk$-theory for 
locally convex algebras  \cite{cuntz}.
  
We would like to remark that 
omitting the generalized inverse corner embeddings  
from \cite[Definition 3.2]{gk}  
has some implications: 

(a) Conjecturally (but not discussed in the sense of a proof),
contrary to \cite{gk}, we may allow the category of all 
algebras or rings, 
because the restriction to {\em quadratik}
rings as in \cite{gk} is 
essentially only necessary for the multiplication map
$A \rightarrow \calk_A(A)$ to be injective, 
to have 
the adjoint $G$-action $\ad(\alpha)$ on $\call_A((A,\alpha))$ available which uses 
multiplication in $A$, 
and as a weaker form of but sufficient substitute for an 
approximate unit occasionally.   

(b) The proof that a functor on $GK^G$-theory is injective if it is injective on the span of homomorphisms, \cite[Theorem 10.8]{gk}, 
breaks {\em completely}
down. 
Similarly so 
for \cite[Proposition 17.6]{gk}, 
that the Baum-Connes map 
is injective if it is injective on the span of homomorphisms. 
The reason is that homomorphisms cannot skip inverse corner embeddings 
as explained in   
\cite[Lemma 8.3]{gk}.  
This 
implicitly stresses the great worth of general `unprojective'  
invertible corner embeddings.

(c) Of course, one has not the generalized Morita 
equivalence,  
\cite[Theorem 14.2]{gk}.  

The proof of $GK^G(\C,A) \cong K^G(A)$ is as follows. 
A key point is to bring simple 
{\em level-one} 
morphisms in $L_1 GK^G(\C,A) \subseteq GK^G(\C,A)$
(each element of $GK^G(A,B)$ is 
a product of level-one morphisms by 
\cite[Lemma 10.2]{gk}) to an even simpler form
called standard form in Proposition  \ref{pstand}, 
to be compared with the standard 
picture of $K$-theory. 
This allows one to form the product 
\be{eq170}
L_1 GK^G(\C,A) \times L_1 GK^G(A,B)
\rightarrow L_1 GK^G(\C,B)
\en 
to simple level-one elements again in Proposition \ref{prop02}, 
and by continuously repeating this argument, we already get 
in Corollary \ref{cor1}  
that each element in  $GK^G(\C,A)$ allows a simple representation 
comparable to those of 
ordinary $K$-theory elements.  
In Theorem \ref{thm42}, we then also 
tediously observe that this procedure plays
the equivalence relations of $GK^G$-theory 
to the simple equivalence relations
of ordinary $K^G$-theory down, yielding the actual main result that
$KK^G(\C,A) \cong K^G(A)$. 

Note 
that in 
Proposition \ref{pstand} we allow all general unprojective 
corner embeddings to be invertible in $GK^G$, and 
Proposition 
\ref{prop02} works partially 
beyond very special $KK^G$-theory.

Note that our strategy of proof 
not only merely shows that $GK^G(\C,A)$ 
has a simple form (as for example 
in Kasparov \cite[§6 Theorem 3]{kasparov1981} 
for $KK^G$), 
but how to compute the general product \re{eq170} 
without the restrictions to level-one, $L_1$, 
algorithmically, 
confer also the 
remark at the end of Section \ref{sec3}.

Let us remark that we proposed 
the {\em possible} concept
also to Kasparov, but we 
had not  
Proposition   \ref{pstand}  
completely
as we had already  
had the idea to rotate one  
homomorphism via a projection,
but not how to rotate the other homomorphism. 
A short while 
later we 
accidentally read the word unitary in another context,
and it was {\em immediately} clear that we should 
trivially complete the projection to 
a {\em unitary} to rotate both 
homomorphisms simultaneously  
and that it will go. This is this note. 

\section{Standard form of $K$-theory}

At the beginning of this section we briefly recall 
$GK^G$-theory in as far as needed for 
the paper, 
but refer for further details to \cite{gk}. 
The key result of this section is Proposition \ref{pstand}. 


We discuss everything in the realm of a small but sufficiently big category of algebras over the field $\C$, but 
might have 
had 
also considered algebras over other fields or even the category of 
quadratik rings as 
in \cite{gk}, 
by complexifying them at first 
to complex algebras before performing ordinary homotopy 
as in \cite{gk}. 

\subsection{$G$-algebras}     \label{sec21}  

Let $G$ be a given discrete group. 
An algebra $A$ is called {\em quadratik} if 
each $a \in A$ 
can be presented as a finite sum of products 
$a_1 b_1 + \ldots + a_n b_n$ in $A$ for $a_i,b_i \in A$. 
Let 
$\Lambda^G$ be a  
{\em small} category where the object class
is a {\em set} of 
quadratik $G$-algebras over the field $\C$ (i.e. algebras $A$ equipped with a group homomorphism $\alpha : G \rightarrow \aut(G)$), 
sufficiently  
closed under all elementary constructions we need throughout
(direct sum, matrix algebra, algebra of compact and adjointable operators etc.), 
and the morphism sets are all $G$-equivariant algebra homomorphisms between 
them.
  For a $G$-algebra $A$, $A^+$ denotes its 
  unitization, and for a homomorphism 
$\varphi:A \rightarrow B$ in $\Lambda^G$, $\varphi^+ : A^+ \rightarrow B^+$ 
its canonical extension in $\Lambda^G$ if nothing else is said.  
Throughout, all homomorphisms and algebras are understood to live in 
$\Lambda^G$. 
The matrix algebras with coefficients in $\C$ 
are denoted by $M_n=M_n(\C)$, where $M_\infty =
\cup_{n \ge 1} M_n$.  
If we want to ignore any $G$-structure or $G$-equivariance 
on algebras or homomorphisms, we sometimes emphasize this by the adjective 
{\em non-equivariant}. 

\subsection{Corner embeddings}

\label{sec22}

A particular feature of $GK^G$-theory inherited from $KK^G$-theory 
is a wider class of corner embeddings 
than the usual matrix 
embeddings  
axiomatically declared to be invertible 
in $GK^G$ in Subsection \ref{subsec21}.(d). 
We always use the adjoint $G$-action on the algebra of $G$-equivariant $A$-module homomorphisms 
on $A$-modules, 
and on subalgebras thereof,
that is, $\hom_A \big ((\cale,S),(\calf,T) \big )= \big (\hom(\cale,\calf), \ad(S,T) \big )$, where 
$\ad(S,T)_g (X) :=   
T_g \circ X \circ S_g^{-1}$  
for $g \in G$.  
Recall from \cite[Definitions 2.1-2.3]{gk} that 
a {\em functional module} $(\cale,\Theta_A(\cale))$ over a $G$-algebra $A$ is a  
$G$-equivariant 
right $A$-module
$\cale$   together with a distinguished $G$-invariant functional space  $\Theta_A(\cale) \subseteq 
\cale^*:=\hom_A(\cale,A)$,  
closed under left multiplication by $A$,
comprised of $A$-linear $G$-equivariant maps from $\cale$ to $A$. 
Often functional modules are supposed to be cofull \cite[Definition 2.13]{gk}. 
The {\em adjointable operators} $\call_A(\cale) \subseteq \hom_A(\cale)$    
($G$-equivariant $A$-linear 
endomorphisms on 
$\cale$) 
is the subset of those 
operators $T$ satisfying $\phi \circ T \in \Theta_A(\cale)$ for all
$\phi \in \Theta_A(\cale)$, see \cite[Definition 2.6]{gk}, and forms a $G$-algebra.     
The {\em compact operators} $\calk_A(\cale) \subseteq \call_A(\cale)$ 
\cite[Definition 2.4]{gk} 
is  the $G$-invariant subalgebra  
of the adjointable operators  
which is the linear span of all  
elementary operators of the form $\Theta_{\xi,\phi}:=\xi \cdot \phi$ for $\xi \in \cale$ and $\phi \in 
\Theta_A(\cale)$. 
A {\em corner 
embedding} is 
a 
$G$-equivariant 
homomorphism $e: (A,\alpha) \rightarrow \calk_{(A,\alpha)} \big ((\cale,S) \oplus (A,\alpha)
 \big )$ defined by $e(a)(\xi \oplus b)= 0_\cale \oplus   
a b$ for all $\xi \in \cale$ and $a,b \in A$, see
\cite[Definition 3.2]{gk}. 

Typically, given a subalgebra $X \subseteq \call_B(\cale)$, its 
unitization can be naturally  
realized by embedding it into $X^+ \subseteq \call_B(\cale \oplus B)$.

\subsection{$GK^G$-theory}

\label{subsec21}

We briefly recall the definition 
$GK^G$-theory from \cite[Definitions 3.5]{gk}, 
a category defined by generators and relations. 
By its definition it is the universal  splitexact, homotopy invariant 
and stable theory for the 
category $\Lambda^G$.
  
%
{\bf The generator set.} 
Object set is the set of objects of $\Lambda^G$.  
Write $\Theta$, the generator morphism set,  for the collection of {\bf (a)} all 
$G$-equivariant 
homomorphisms $\varphi:A \rightarrow B$ between all $G$-algebras 
$A$ and $B$ of $\Lambda^G$, and {\bf (b)} the {\em synthetical} morphisms 
$e^{-1} :  
\calk_A(\cale \oplus A) \rightarrow A$ for each 
corner embedding $e : A \rightarrow \calk_A(\cale \oplus A)$  
of $\Lambda^G$, 
and  
{\bf (c)} the {\em synthetical} splits $\Delta_{S}: M \rightarrow J$ (denoted sloppily by $\Delta_s$) 
for each {\em splitexact sequence} in 
$\Lambda^G$, called $S$, 
\be{eq32}
\xymatrix{ 
  0  \ar[r]  &  J \ar[rr]_i   
  &&  M    
\ar@<.5ex>[rr]^f   \ar@<-.5ex>[ll]_{\Delta_{s}}   
 &&  A   \ar[ll]^{s} 
 \ar[r]   & 0  
     }
     \en
     (here $\Delta_s$ is not part of the definition of the diagram, 
     but notated for purposes of anticipation). 
     {\bf The free morphisms.} 
Then denote by $L$ the set of all valid finite expressions 
(= strings) comprised
of letters from $\Theta$, a product sign $\cdot$ (usually omitted)
and plus $+$ and minus $-$ signs, and brackets $($ and $)$,
in the sense that source and range objects of the letters of $\Theta$
fit together, and the expressions are valid in the usual additive   
category theoretical sense with respect to brackets and the algebraic operations. 
 
The order of writing in $L$ (and in $GK^G$) is from left to right, that is, morphisms 
$f:A \rightarrow B$ and $g: B \rightarrow C$ are composed 
to $fg$.

 {\bf The quotient by relations.} 
 Then the additive category $GK^G$ 
 ({\em Generators and relations $KK$-theory with $G$-equivariance})
 is defined to be the set-theoretical 
 quotient 
 $L/R$, where $R$ is the 
 equivalence relation on $L$ 
 induced by {\bf (a)} the associativity and distribution laws 
 with respect to brackets and the algebraic operations,
 {\bf (b)} the composition
law 
$\varphi \circ \psi  \equiv_R \psi \varphi$ 
for composable 
homomorphisms  $\varphi$ and $\psi$,  
{\bf (c)}
the one-law that the 
homomorphism $\id_A: A \rightarrow A$ 
is a one-sided unit 
for each generator in $\Theta$ starting or ending at $A$,  
{\bf (d)}
the law that each corner embedding $e: A \rightarrow \calk_A(\cale \oplus A)$  
is invertible with inverse (synthetical) morphism $e^{-1} \in \Theta$,
{\bf (f)}
the splitexactness axioms 
$1_M \equiv_R \Delta_s i + f s$ and 
$1_J  \equiv_R i \Delta_s$ for each splitexact
sequence $S$ as stated in \re{eq32}, and 
{\bf (g)} homotopy invariance, that is
$\varphi_0 \equiv_R \varphi_1$ for 
each homotopy 
homomorphism 
$\varphi:A \rightarrow B[0,1]$, 
where $B[0,1]:= (B,\beta) \otimes (C([0,1]),\mbox{triv})$.  
%

We usually 
express elements of $GK^G$ by writing 
representants of $L$ down and do not use class brackets.

\subsection{Extended double splitexact sequence}

\label{subsec22}

We briefly recall the concept of extended splitexact sequences in $GK^G$,
see \cite[Definitions 4.1, 5.1 and 9.1]{gk}. 

{\bf (a)}
A typical building block of 
$GK^G$ 
used throughout from which all morphisms can be composed is the morphism
and product of elements in $\Theta$ 
\be{eq33} 
(s_+ \oplus \id_A) \cdot \;f_2  \cdot f_1^{-1} \; \cdot \Delta_{s_- \oplus \id_A} \;  \cdot e^{-1} 
\en
in $GK^G$,   
sloppily denoted by
$s_+ \na_{s_-}$ for short, associated to the following diagram 
of morphisms 
in the category $GK^G$ 
called {\em extended double splitexact sequence}, 
\be{eq35}
\xymatrix{ 
B     \ar[r]^-e&  J \ar[rr]_-i   
  &&  X    
\square_{s_-  \oplus \id_A} A     
\ar@<.5ex>[rr]^-f   \ar@<-.5ex>[ll]_-{\Delta_{s_- \oplus \id_A}}   
 &&  A   \ar[ll]^-{s_\pm \oplus \id_A} 
     } ,	
     \en
    
 where without the 
corner embedding $e$, the above diagram is a splitexact sequence 
as in \re{eq32} with respect to the 
split $s_- 
\oplus \id_A$, where the zeros at the ends are 
technically 
there but not drawn.


{\bf (b)} 
More precisely, 
here $B,J,X$ and $A$ are $G$-algebras. 
Further, $s_-: (A,\alpha) \rightarrow (X,\gamma_-)$ and 
$s_+: (A,\alpha) \rightarrow (X,\gamma_+)$
are two given equivariant homomorphisms.  
Note that $X$ is equipped with two $G$-actions, $\gamma_-$ and
$\gamma_+$.  
Further, $s_\pm \oplus \id_A: A \rightarrow (X \oplus A, \gamma_\pm \oplus \alpha)$
denote the two equivariant canonical homomorphisms 
$(s_\pm \oplus \id_A)(a) := s_\pm(a) \oplus a$, 
sloppily 
also when the range algebra $X \oplus A$ is restricted
to a subalgebra (as for example everywhere in diagram \re{eq35}) . 

Moreover, $J$ is a two-sided $G$-invariant ideal in $(X,\gamma_-)$,
and $i : 
J \rightarrow (X,\gamma_-)$ denotes its equivariant 
embedding. Here, $e$ is an equivariant  corner embedding.




{\bf (c)} 
One also requires that 
$s_+(a) -s_-(a) \in J$ for all $a \in A$.
%
Consequently, the two 
equivariant {\em subalgebras} 
of $X \oplus  A$, 
$$M_\pm  :=   X \square_{s_\pm   \oplus  \id_A} A  
\; := \quad J \oplus 0  \; +  \;(s_- \oplus \id_A)(A)
\quad \subseteq  \quad (X \oplus A, \gamma_\pm \oplus \alpha)$$ 
are the same {\em non-equivariantly}, that is, $M:= M_- = M_+$.
%
Finally, 
$f: M_- \rightarrow A$ is the equivariant canonical projection 
$f(x \oplus a) = a$ onto $A$.   

Note that independent of the size of the 
auxiliary algebra $X$, 
$M$ is always the same, namely {\em linearly} isomorphic to the direct sum
$J \oplus A$. 

{\bf (d)} 
Without the corner embedding $e$, the above diagram 
is {\em non-equivariantly} 
 a double split exact sequence with two splits
 $s_\pm \oplus \id_A$. 
To make it - 
informally speaking - equivariantly double split,
the canonical equivariant corner embeddings 
$f_1,f_2$
$$\xymatrix{ M_- 
 \ar[rr]^{f_1} & &  (M_2(M), \delta)
 & &   M_+  \ar[ll]^{f_2}}
$$ 
into the top left and bottom right corners, respectively, 
are 
assumed to be provided. 
 The task of $f_2 \cdot f_1^{-1}$ in formula \re{eq33} is
 to get a path from $M_+$ to $M_-$, thereby changing 
 the $G$-action from $\gamma_+$ to $\gamma_-$. 
 
 The last $\delta$ is called the {\em $M_2$-action} of line 
 \re{eq35}. 
 
 {\bf (e)} 
 From now on we sloppily abbreviate $s_\pm \oplus \id_A$
 by $s_\pm$ in diagrams and also 
 in subscripts of $\Delta$ and $\square$. 
 That means, from now on we shall replace every occurrence 
 of $s_\pm \oplus \id_A$ by $s_\pm$ in a diagram like \re{eq35}. 
  
{\bf (f)}
  Finally, 
  we have also entered the synthetical split $\Delta_{s_- \oplus \id_A}$ in the above diagram \re{eq35}, which may be heuristically viewed 
  as the canonical {\em linear} split $u: M_- \rightarrow J$ which projects onto $J$. 
  A crucial point of $GK^G$-theory then is that we have that
  synthetical split available as a morphism, making this theory 
  splitexact.  
  
  \subsection{}
  {\bf Level one morphisms.} 
  The subset of all morphisms \re{eq33} of $L$ as discussed above is denoted
  by $L_1 \subseteq L$. They are called {\em level-one 
  morphisms}  
  both in $L$ and as canonical representants in $L/R$.
The canonical image of $L_1$ in $GK^G$ is also denoted by $L_1 GK^G$. 
  
  {\bf Correspondence between morphisms and diagrams.} 
  Informally, two different diagrams \re{eq35} 
  are called { equivalent} if their associated elements
  \re{eq33} are identical in $GK^G$.  
  It is much easier to work with pictorial diagrams like \re{eq35} than
  with formal algebraic expressions like \re{eq33}, whence we often
  use the 
diagram style and its formal expression interchangeable,
  often without much saying. 
  It is understood that diagrams like \re{eq35} {\em are} extended splitexact
  sequences, even if $\Delta$ is not entered.




\subsection{Presenting 
							generators}  \label{section26}

In $GK$-theory we may present each generator of $\Theta$
by a level-one morphism in $L_1$: 
 
\begin{lemma}			\label{lem1}

The generators $\Theta$  
of $GK^G$-theory can 
be 
expressed as level-one morphisms 
as stated in formula \re{eq33} 
in $GK^G$-theory as follows. If $\varphi:A \rightarrow B$ is a homomorphism, $e^{-1}: J \rightarrow A$ an inverse  corner embedding, and $\Delta_{s}: M \rightarrow J$ 
a synthetical split associated to 
the splitexact sequence \re{eq32}, then their level one morphism correspond to the following diagrams, respectively,
\be{eq10}
\xymatrix{  
B     \ar[r]^\id &  B \ar[r]     & B    
\square_{0 } A    \ar@<.5ex>[rr] 
  &&  A   \ar[ll]^{\varphi,\; 0} 
 & (\varphi)  ,
     }
     \en
\be{e11}
\xymatrix{   A  \ar[r]^e    
 &  J  \ar[r] & J 
\square_{0 } J   
\ar@<.5ex>[rr]  
 &&  J  \ar[ll]^{\id,\; 0} 
& (e^{-1} )  ,
}  \en
\be{e12}
\xymatrix{    
J  \ar[r]^\id  &  J  \ar[r]^{i \oplus 0}  & M 
\square_{s \circ f } M   
\ar@<.5ex>[rr]
 &&  M  \ar[ll]^{\id ,\; s \circ f} 
 &   ( \Delta_s  )  .
}
\en

(For example, $\varphi = \varphi \Delta_0 \id^{-1}$, the corner embeddings  $f_1,f_2$ 
can be 
omitted, or set identical $f_1 =f_2$,  in all diagrams above.)
\end{lemma}

\begin{proof}

The first line is by 
\cite[Lemma 7.1]{gk}, 
the second one 
trivially by the first line for $\varphi = \id_J$ 
composed 
with $e^{-1}$, and the third line by
  \cite[Lemma 7.4]{gk}.
\end{proof}

\subsection{Standardform of level-one $K$-theory}

A key is to 
bring 
level-one elements starting at $\C$ 
to a particular
simple 
``standard form"  in $GK$-theory, best to be compared
with the standard form in ordinary $K$-theory for
operator algebras: 

\begin{proposition}[Standardform of level 
one $K$-theory elements]   
\label{pstand}

Every level-one 
morphism $s_+ \nabla_{s_-} \in L_1 GK^G(\C,B)$ can be presented as 
a level-one morphism $t_+ \nabla_{t_-} \in  L_1 GK^G(\C,B)$ 
as indicated in 
the diagram
$$\xymatrix{ 
B   \ar[d]^\id   \ar[r] &  J 
\ar[r]  \ar[d]^{e_{11}}   & X
\square_{   s_-  
 } \C    \ar[d]   
 \ar@<.5ex>[r]     
&  \C   \ar[l]^{s_\pm}    \ar[d]    \\
B  \ar[r]  &  M_2(J)  \ar[r] &	M_2 (X^+) 
\square_{   t_-    
} \C  
 \ar@<.5ex>[r]  
  &  \C   \ar[l]^-{t_\pm}      
}$$
%
where  
\begin{eqnarray*}
t_-(\lambda)  &=& 
\big ( s_- \oplus 
q \big) \cdot 
\ad (U_{\pi/2}) \;(\lambda) 
 \;=  \; \lambda  \big(0_{X^+} \oplus 1_{X^+} \big)  
\qquad (\lambda \in \C),     \\
t_+ 
&=&   \big ( s_+ \oplus 
q \big) \cdot 
\ad (U_{\pi/2})  
=
\left ( \begin{matrix} 
p^\bot  s_+  p^\bot  
&  -  p^\bot s_+ p  \\
-p  s_+  p^\bot  &   p s_+ p + q  \\ 
\end{matrix}\right ). 
%
%
\end{eqnarray*}

Observe the point that $t_-$ is so simple. 
(Hence, one may finally replace $X^+$ by $J^+$ in the above diagram.) 

\end{proposition}

\begin{proof}

The proof goes by adding on a trivial element and then performing
a rotation homotopy.  

{\bf (a)}  
%
%
%
Set $p:= s_-(1_\C ) \in X$,  $p^\bot := 1_{X^+} - p \in X^+$, and define $q: \C \rightarrow X^+$ 
by 
$$q(\lambda)= \lambda p^\bot . $$		

%
Consider the trivial zero element $0=q \na_q \in L_1 GK^G(\C,B)$ 
(see \cite[Remark 3.6.(iii)]{gk}),    
where the $G$-action in the range algebra 
$X^+$ of $q$ is chosen to be 
the one of that 
of $s_-$, 
and get 
$v_+ \na_{v_-}  : = 
 (s_+ \oplus q) \na_{ s_- \oplus q}   
=  s_+ \na_{s_-} +  q \na_q  = s_+ \na_{s_-}  $ 
in $GK^G$ by the 
addition lemma \cite[Lemma 9.9]{gk} 
for the 
homomorphisms   
$v_\pm:\C \rightarrow   M_2(X^+)$  
defined by
$v_\pm := s_\pm \oplus q$.

{\bf (b)} 
In $v_+ \na_{v_-}$ we now want to rotate the projection $p$ from  the first coordinate of $v_-$ into the defect $p$ of $1-p$ 
of the second coordinate of $v_-$,
and to this end define a  homotopy $U:
[0,\pi/2] \rightarrow M_2(X^+)$   
of invertible operators 
by 
\be{rotu}
U_t := \left ( \begin{matrix} 
\cos(t) p  
&  \sin(t) p  \\
-\sin(t) p &  \cos(t) p   \\  
\end{matrix}\right )
+ 
\left ( \begin{matrix} 
p^\bot
&  0  \\
0 &  p^\bot   \\  
\end{matrix}\right )  .		
\en

{\bf (c)} 
More precisely we define a homotopy in $L_1$ by  
$$\xymatrix{ 
B[0,\pi/2]      \ar[r] &  M_2(J)[0,\pi/2] \ar[r]     
& \Big ( M_2(X^+)[0,\pi/2]  \Big )      
\square_{V_- } \C      \ar@<.5ex>[rr] 
   &&  \C   \ar[ll]^-{V_\pm}      
}   , $$
where $V_\pm : \C \rightarrow M_2(X^+)[0,\pi/2]$ is given by
the family of maps  
$V_\pm^{(t)} =  
\ad (U_t) \circ v_\pm : \C \rightarrow M_2(X^+)$ 
for $t \in [0,\pi/2]$. 

If $\gamma: G \rightarrow \aut \Big [  M_2 \Big ( M_2( X^+)  
\square_{v_- } \C \Big )  \Big ]$ is the $M_2$-action 
of $v_+ \na_{v_-}$, then the $M_2$-action  
of the $L_1$-homotopy is set to be
$\Gamma: G \rightarrow \aut  \Big [  M_2 \Big ( \Big ( M_2( X^+)[0,\pi/2] \Big ) 
\square_{V_- } \C  \Big ) \Big ]$ with point of time actions 
($t \in [0,\pi/2])$)
$$\Gamma^{(t)} = 
\Big (  \big  ( \ad (U_t) \oplus \id_\C \big )  
\otimes \id_{M_2} \Big )
\circ \gamma \circ  
\Big (  \big (\ad(U_t^{-1}) \oplus \id_\C  \big ) 
\otimes \id_{M_2} \Big )   .
$$

{\bf (d)} 
Then we finish the proof by applying this homotopy, that is,  
by applying 
the algebraic analogy of \cite[Lemma 8.10]{aspects} 
we obtain 
$$v_+ \na_{v_-} = (\ad(U_0) \circ v_+  ) \na_{  \ad(U_0) \circ v_-}
= ( \ad(U_{\pi/2}) \circ v_+  ) \na_{\ad(U_{\pi/2})   \circ v_-} 
=: t_+ \na_{t_-}  , $$
where $t_\pm$ 
are obviously defined. 
 %
%
%
\end{proof}

\section{$K$-theory in special $GK^G$-theory}

\label{sec3}

In this section we exploit the simple standard form of Proposition 
\ref{pstand} to multiply such level-one $K$-theory elements 
with any other level-one element in $GK^G$-theory 
in 
Proposition \ref{prop02}. 
However, up to now we can 
achieve 
this only for ordinary corner 
embeddings  
$e : B \rightarrow J=M_n(B)$ with special $G$-action, 
and 
by  
restricting the $G$-action further  
and axiomatically declaring only 
those corner embeddings  to be 
invertible 
we get a descent
(very special) $GK^G$-theory, also 
incorporating the validity of a Green-Julg theorem. 

 For 
	isomorphisms $M_n(A) \cong \calk_A(A^n)$ 
	and $M_n(\call_A(A)) \cong \call_A(A^n)$  
see 
\cite[Lemmas 2.22 and 2.23]{gk}.  
Recall from Subsection \ref{sec22} that 
$\calk_A(A^n)  \subseteq \call_A(A^n)$.

\begin{definition}		\label{def31}

{\rm 

A corner embedding is {\em special} 
if it is of the 
 usual form $e : (A,\alpha) \rightarrow (M_n(A),\delta)$ ($1 \le n \le \infty$)
  with $e(a) = a \oplus 0 \oplus 0 \oplus \cdots$ and 
  can be extended to a usual equivariant corner embedding 
  $e^+: (A^+,\alpha^+) \rightarrow (M_n(A^+), \delta_2)$ 
for a 
$G$-action $\delta_2$ extending $\delta$.  
%
  It is {\em very special} if it is even of the form 
  $e : (A,\alpha) \rightarrow (M_n(\C),\gamma) \otimes (A,\alpha)$ 
  ($1 \le n \le \infty$). 

}
\end{definition}

\begin{lemma}			\label{lem31} 

Let $(A,\alpha)$ be a $G$-algebra. 
Without any $G$-actions, let $B$ be an algebra between $M_n(A)$ and $M_n(\call_A(A))$. 
Assume now $(B,\Gamma)$ is a $G$-algebra such that its upper left corner 
embedding 
$(A,\alpha) \rightarrow (B,\Gamma)$ is $G$-invariant.  
Then:

{\rm (i)}  
There is a $G$-module action $\gamma$ on the $(A,\alpha)$-module $A^n$ such that
$$(M_n(A),\Gamma|_{M_n(A)}) \cong \calk_{(A,\alpha)} ((A^n,\gamma))  \subseteq (B,\Gamma)    $$ 
is a $G$-invariant subalgebra. 

{\rm (ii)} $(B,\Gamma) \subseteq \call_{(A,\alpha)}((A^n,\gamma))$ is
a $G$-invariant subalgebra. 

{\rm (iii)} The $G$-action $\Gamma$ on $B$ is uniquely 
determined by its restriction 
to $M_n(A)$ (and even just the first column of it).



\end{lemma}

\begin{proof}

Identify $A^n \subseteq M_n(A)$ as the first, left-most column, and $(A,\alpha) \subseteq
M_n(A)$ as the upper left corner. 
Then $A^n$ is the ordinary $A$-module by performing multiplication in $M_n(A)$ 
under the last identifications.  
We equip $A^n$ with the $G$-action $\gamma$ defined to be $\Gamma$ restricted to $A^n$. Indeed, 
using that $A$ is 
quadratik, see Subsection \ref{sec21},  
the 
$G$-invariance of $A^n$  follows from
\begin{eqnarray*}
\Gamma(a_{i1}  b_{11} \otimes e_{i1} )  & = &  
\Gamma \big ((a_{i1} \otimes e_{i1}) (b_{11} \otimes  
e_{11}  \big )
= \Gamma( a_{i1} \otimes e_{i1} )  \Gamma( b_{11} \otimes e_{11} ) \\
	&=&  \Gamma( a_{i1} \otimes e_{i1} )  \big (\alpha( b_{11}) \otimes e_{11}  \big ) 
\subseteq A^n   .   
\end{eqnarray*} 

This computation also shows that $\gamma$ is a $G$-module action for $(A^n,\gamma)$ over $(A,\alpha)$.

The analogous computation shows that the first upper line $L \cong A^n \subseteq M_n(A)$ 
is 
$\Gamma$-invariant. 
As every element of $M_n(A)$ can be written as a sum of products of the form
$x y \in M_n(A)$ with $x \in A^n, y \in L$, and $\Gamma(x y) = \Gamma(x) \Gamma(y) 
\subseteq  A^n \cdot L \subseteq M_n(A)$, $M_n(A)$ is $\Gamma$-invariant. 

We define the $G$-equivariant map $\phi: (M_n(A),\Gamma) \rightarrow \calk_{(A,\alpha)}
((A^n,\gamma))$ by matrix-vector multiplication $\phi(a)(b)= a \cdot b \in A^n$ 
($a \in M_n(A), b \in A^n$)  
to be performed as the ordinary product in $M_n(A)$, where $A^n \subseteq M_n(A)$ 
is identified as above. 
Then 
$$\phi \big (\Gamma(a) \big ) (b) = \Gamma(a) \cdot b = \Gamma(a) \cdot \Gamma \big (\Gamma^{-1}(b) \big ) 
= \gamma \big ( a \cdot \gamma^{-1}(b) \big )
= \big (\ad(\gamma)(\phi(a) \big ) 
 (b)  ,$$
which shows $G$-equivariance of $\phi$.

(ii)
Analogously as $\phi$, we define $\psi:(B,\Gamma) \subseteq M_n(\call_A(A))
\rightarrow \call_{(A,\alpha)} ((A^n,\gamma))$ 
by matrix-vector multiplication 
$\psi(a)(b)= a \cdot b \in A^n$ 
($a \in M_n(\call_A(A)), b \in A^n \subseteq \call_A(A)^n$)   performed in  $M_n(\call_A(A))$, 
in which $A^n$  
is identified as the first column again. 

By the same computation as for $\phi$, we see that $\psi$ is $G$-equivariant. 

(iii) 
We have seen in (i) that $\gamma$ is just the restriction of $\Gamma$ to the 
first column of $M_n(A)$. But reversely it determines $\Gamma$ completely 
by (ii). 
%
 %
%
%
%
\end{proof}

Lemma \ref{lem31}.(i) applied to $M_n(A)$ and $M_n(A^+)$ tells us that 
the special corner embeddings $e$ and $e^+$ of 
Definition \ref{def31} are indeed valid, invertible  
corner embeddings in the sense of Subsection \ref{sec22}, 
and Lemma \ref{lem31}.(ii) applied to $M_n(A)$ that  
the special requirement of $M_n(A)$ is that the subalgebra $M_n(A^+)
\subseteq M_n(\call_A(A))$ is $G$-invariant and Lemma \ref{lem31}.(iii) 
that its $G$-action is then unique. 

If $A$ is unital, then $e$ is always special because $M_n(A^+) \cong M_n(A) \oplus M_n(\C)$ can be extended with the trivial action on the second summand.  

$GK^G$-theory is called {\em very special} if only very special 
corner 
embeddings are axiomatically declared to be invertible 
in Subsection \ref{subsec21}.(d). 
%
An $M_2$-space $M_2(M)$ can then only be trivial in the sense 
that 
its $G$-action  is 
of the form $\delta = 
\gamma \otimes \id_{M_2}$ for 
$(M,\gamma) = (M_-,\gamma_-) = (M_+,\gamma_+)$ 
by Subsection \ref{subsec22}.(d),  
and thus 
the corner 
embeddings $f_1,f_2$ in Subsection \ref{subsec22} can be omitted as anyway $f_2 \cdot f_1^{-1} = \id_{M}$ in $GK^G$ 
by an ordinary rotation homotopy.

%
Important invertible corner embeddings in very special $GK^G$-theory
include those of the form $e:A \rightarrow 
%
\big (\entd_{\C}(X),\ad(\gamma) \big ) \otimes (A,\alpha)$, where $X$ is a vector space equipped with a $\C$-linear $G$-action $\gamma$, 
such that $\entd_{\C}(X) \cong M_n(\C)$. 
For $G$ a finite group and $X:= \ell^2(G)$ this is the corner 
embedding 
(the averaging map $1 \otimes f$ of \cite[Lemma 16.2]{gk}) 
sufficient for the proof of the Green-Julg theorem, so that
we 
note:

\begin{lemma}		\label{lemma31} 

The Green-Julg theorem 
\cite[Corollary 16.11]{gk} for finite discrete groups $G$ holds in very special $GK^G$-theory. 

\end{lemma}

We can now fuse standard form $K$-theory $L_1$-elements with any 
level-one     
elements to $L_1$-elements   
in the following 
proposition.  

Its proof relies also on a 
series of applications of elementary results 
from \cite{gk}, but taking care about the involved $G$-actions
makes it somewhat 
lengthy   
(including the necessity of Lemma \ref{lem31}).  
 For very special $GK^G$-theory 
the proof is much simpler, see the 
last paragraph of the proof.   

\begin{proposition}[Fusion of 
standard 
$K$-theory element with 
general $GK$-theory element]

\label{prop02}

A given level one { standard} $K$-theory element $p_+ \na_{p_-} 
\in L_1 GK^G(\C,A)$ 
(represented as the right most edge of the diagram below)   
can be multiplied with any given level one $GK$-theory element
$s_+ \na_{s_-} \in  L_1 GK^G(A,B)$ 
(represented as the 
top edge of the diagram below) to a level one $K$-theory element 
\be{eq32b}
t_+ \na_{t_-}   
= p_+ \na_{p_-}  \cdot s_+ \na_{s_-} 
\en
in $L_1  GK^G(\C,B)$ 
(represented as the bottom edge of the diagram below) 
in $GK^G$-theory 
as indicated in the diagram
$$\xymatrix{
B  \ar[r]^e  \ar[d] & J \ar[r] 
\ar[d]  
& X \square_{s_-}  
A   \ar[d]  
\ar@<.5ex>[r]
& A \ar[l]^{s_\pm 
} \ar[d]^e    \\
B  \ar[r]^h     
\ar[d]  & M_n(M_m(J)) \ar[r]^{j_1} 
\ar[d] & M_n(M_m(X)) \square_{ v_- }   
E     \ar[d]    
\ar@<.5ex>[r]    
& M_n(M_m(A)) \ar[l]^{v_\pm 
 }    
\ar[d]_{i}  \\  
B  \ar[r] 
\ar[d] 
&   
    M_{2n  m}(J) \ar[r]^{j_2}  \ar[d]  & M_{2n}(M_m(X)^+) \square_{u_-}   
D_-   
\ar@<.5ex>[r] 
\ar[d]   
 &   
M_{n}(M_m(A)^+) \square_{p_-} \C   
  \ar[l]^{ u_\pm   
  }   
  \ar@<-.5ex>[u]_{\Delta_{p_-}}  
  \ar[d]_{f_1}
  \\
B  \ar[r] 
\ar[d] 
& 
	M_{4n m } (J)  \ar[r]  \ar[d]  & M_{4n}(M_m(X)^+) \square_{z_-}   
M_2(D)     
\ar@<.5ex>[r]   
 &   
      M_2(D)    
  \ar[l]^{ z_\pm   
  }  \\
B  \ar[r] 
\ar[d] 
& 
	M_{2n m }( J) \ar[r]  \ar[d]  & M_{2n}(M_m(X)^+) \square_{y_-}   
D_+   
\ar@<.5ex>[r]   \ar[u]   
 &   
M_n(M_m(A)^+) \square_{p_+} \C   
  \ar[l]^{ y_\pm   
     }  
  \ar[d]_{g}  
  \ar[u]_{f_2}
  \\
  B  \ar[r] 
  & 
	M_{2n m}( J)  \ar[r] &  M_{2n}(M_m(X)^+) \square_{t_-}   
 \C   
\ar@<.5ex>[r]  
		\ar[u]    
 &   
\C  \ar@<-.5ex>[u]_{p_\pm}   
\ar[l]^-{ t_\pm:=
  \hat s_\pm   \circ p_+  \oplus \hat s_\mp \circ p_- }	   
}$$
provided  
the $M_2$-action of the first 
line and all corner 
embeddings  of the last most right row of the diagram  
are special,   
and the matrix sizes $m$ and $n$ are finite.  
Also, 
$e$ is assumed to be 
the obvious composition of two 
special corner embeddings.  

Here, 
for any 
homomorphism $s : A \rightarrow X$  we 
define $\hat s : M_n(M_m(A)^+) \rightarrow  M_n(M_m(X)^+)$ by the formula 
\be{eq36}
\hat s:=   ( s \otimes \id_m)^+ \otimes \id_n   ,
\en
which does not determine $\hat s$ in $GK^G$ completely as we use this 
map with 
various $G$-actions in its domain and range.  

Moreover, all $M_2$-actions in the diagram remain special.
In very special $GK$-theory, the whole diagram  
and computation stays within that theory, and infinite matrix sizes $m$ and
$n$ are allowed.

\end{proposition}

\begin{proof}

{\bf Line 1 and last row:} 
 Let us explain this diagram. 
 The first line represents a given date, 
namely $s_+ \na_{s_-}$. 
 The right most row represents the given standard $K$-theory element $p_+ \na_{p_-}$. 
Here, the algebras are $E:= M_n(M_m(A))$ and 
$D_\pm:= M_n(M_m(X)^+) \square_{p_\pm}   
 \C   $  for short. 
 Recall 
 from Subsection \ref{subsec22}.(c)       
 that non-equivariantly we have $D:= D_- = D_+$.    
 In the 
 diagram we have wrongly drawn the arrows $g$ and $p_-$
 between $\C$ and $D_+$ for better readability, but 
 in reality they 
 stretch further between $\C$ and $D_-$. 

{\bf Line 2:} 

Our aim is to show that the first line of the above diagram multiplied by $e^{-1}$ 
is the second line of the above diagram. 

{\bf (a)} 
At first bring $s_+ \na_{s_-}$ to 
a form in $GK^G$,
such that we may 
assume that $X \subseteq  \call_J(\cale)$ 
(for example $\cale = J$) and $J = \calk_J(J)$, see \cite[Lemma 7.3]{gk}.  

Denote by $M := X \square_{s_-} A 
\subseteq	  
 \call_J(\cale ) \oplus  A$ the non-equivariant subalgebra,  
see Subsection \re{subsec22}.(c).     
By \cite[Lemma 7.3]{gk},    
the $M_2$-space of $s_+ \na_{s_-}$  
has a $G$-action by restriction as follows, 
\be{eq100}
(M_2(M),\delta) \subseteq  \Big (M_2 \big (\call_J(\cale) \oplus 
A  \big ), \ad (S \oplus  \alpha \oplus T \oplus \alpha) \Big  ), 
\en
where $S,T$ are two 
$G$-actions on the 
$J$-module $\cale$ ,
and $(A,\alpha) \cong (\calk_A(A), \ad (\alpha))$.  

{\bf (b)}   	
Our first objective is to 
replace the first two lines of the above diagram by their unital versions  
(for $n=1$ for the moment)  
$$\xymatrix{
B  \ar[r]^e  \ar[d] & J \ar[r] 
\ar[d]  
& X^+ \square_{s_-^+}  
A^+   \ar[d]  
\ar@<.5ex>[r]
& A^+ \ar[l]^{s_\pm^+ 
} \ar[d]^{e^+}    \\
B  \ar[r]^h     
  & M_m(J) \ar[r]^-{j_3} 
 & M_m(X^+) \square_{ r_- } 
M_m(A^+)    
\ar@<.5ex>[r] 
& M_m(A^+) \ar[l]^-{r_\pm
 }   . } $$
%

Indeed, if $s_+ \na_{s_-}$ is carefully written
out as in \re{eq33}, then the first line of the last
diagram is 
the valid morphism 
\be{eq102}
(s_+^+ \oplus \id_{A^+}) \cdot \;f_2^+  \cdot {f_1^+}^{-1} \; \cdot \Delta_{s_-^+ \oplus \id_{A^+}} \;  \cdot e^{-1}   , 
\en
where the unitizations 
$f_1^+ ,f_2^+: M^+ \rightarrow (M_2(M^+))$  
of the - by assumption special - corner 
embeddings $f_1,f_2$ 
exist by 
Definition \ref{def31}. 

Thereby observe that if $M := X \square_{s_-} A$ then indeed $X^+ \square_{s_-^+}  
A^+ \cong M^+$ is just the 
unitization of $M$.  
 Also observe that $s_+^+(1) -s_-^+(1) = 1-1 =0 \in J$ 
as required in Subsection \ref{subsec22}.(c). 
 For the canonical 
 embedding $\iota_A: A \rightarrow A^+$, by \cite[Lemma 7.2]{gk} we get 
 \be{eq105}
 \iota_A \cdot s_+^+ \na_{s_-^+} = (s_+^+ \circ \iota_A) \na_{s_-^+  \circ  \iota_A}  
	= s_+ \na_{s_-}     . 
 \en

{\bf (c)} 

Note that the $G$-action notated in \re{eq100} 
is defined on $M_2(\call_J(\cale) \oplus A)$, 
and thus also on the bigger unital algebra 
  $M_2 \big (\call_J(\cale \oplus 
B) \oplus A^+ \big)$ 
(by an obvious $G$-action), 
which contains also $M_2(M^+)$ non-equivariantly, 
but a typical general question is if the subalgebra $M_2(M)$ is invariant under it (even if  we know it is in our current example),  
what we want to inspect now. 

This is 
clarified  
by the 
equivalent equivariance condition of the space $M_2(M)$ 
in $s_+ \na_{s_-}$  
in \cite[Lemma 6.5]{gk}, 
namely 
one needs $s_-(a)(S_g T_g^{-1} - 1_{\call_B(\cale)}) \in \calk_B(\cale)$ 
for all $a \in A, g \in G$ (and 
similarly so for $T_g S_g^{-1}$) 
in order 
that $M_2(M)$ is $G$-invariant,
and for $G$-invariance of $M_2(M^+)$ 
in $s_+^+ \na_{s_-^+}$, 
which holds if $f_1,f_2$ are special in 
$s_+ \na_{s_-}$ 
by (b) - and by Lemma \ref{lem31}.(iii) the $G$-action 
of $M_2(M^+)$ is 
unique and thus it must be invariant under 
the  considered 
$G$-action of \re{eq100} -,    
this means 
\be{eq55} 
S_g T_g^{-1} - 1_{\call_B(\cale)} \in \calk_B(\cale)   ,
\en 
where we have omitted the now 
ineffective factor $1_{\call_B(\cale)} = s_-^+(1_{A^+})$.

Actually one had to add a summand $B$ to $\cale$ at first to handle $M^+$ and thus the criterion 
$(S_g \oplus \beta_g) (T_g^{-1} \oplus \beta_g^{-1}) 
-   
1_{\call_B(\cale \oplus B)} \in \calk(\cale \oplus B)$, but this is obviously 
equivalent to \re{eq55}, where $(B,\beta) =(\calk_B(B),\ad(\beta))$. 

On the other hand, 
the 
validity of condition \re{eq55} for the $M_2$-space $M_2(M)$ 
of $s_+ \na_{s_-}$ 
always
shows that the corresponding corner embeddings $f_1,f_2$ are special 
because 
by invariance condition \cite[Lemma 6.5]{gk}  again this shows that
not only $M_2(M)$ but also  $M_2(M^+)$ 
is invariant and 
thus the above 
unitization 
\re{eq102} can be 
formed.

{\bf (d)} 
%
We have 
proven in (c) that - completely in general - corner embeddings $f_1,f_2$ 
at the $M_2$-space 
as in \re{eq100}  in a morphism $s_+ \na_{s_-}$, in details \re{eq33},  are special if and only if the invariance condition
\re{eq55}  holds.

{\bf (e)}		
But this also implies that if we restrict $s_\pm$ to a subalgebra,
so consider the morphism $z:= (s_+ \circ \varphi) \na_{s_- \circ \varphi}$
for a 
homomorphism $\varphi: A' \rightarrow A$   
by \cite[Lemma 7.2]{gk}
then  
the $M_2$-action of $z$ is special again because 
the strong invariance condition \re{eq55} 
persists (also for subspaces $M_2(N) \subseteq M_2(M)$ 
and the embedding \re{eq100} remains).  

{\bf (f)}		
Then one 
applies 
\cite[Proposition 9.7]{gk} 
(fusion of an inverse corner embedding 
with an extended splitexact sequence)   
to obtain the second line of the last diagram   with 
$${e^+}^{-1} \cdot s_+^+ \na_{s_-^+}     
=  r_+ \na_{r_-}	  .$$ 
 
 After tedious inspection of proof there, it turns out that
 canonically 
 we have
	$r_\pm  
	:= s_\pm^+ \otimes \id_m$  
	and $j_3$ 
	is the canonical embedding $j_3(j) = j \oplus 0$.
Note that these are the non-equivariant formulas, but the $G$-actions 
in the range of $r_+$ and $r_-$ are 
different and influenced by $s_+^+$ and $s_-^+$, respectively. 

Since we assume that the matrix size $m$ is finite, 
the map $r_-$ of the second line of the above diagram is unital, 
and thus the $M_2$-space of that line is unital, and thus special.

Applying \cite[Lemma 7.2]{gk}, 
similarly as in \re{eq105} we get $\iota  \cdot 
r_+ \na_{r_-} = (r_+  \circ \iota) \na_{r_- \circ \iota}  = :  v_+ \na_{v_-}$ 
(where $\iota: M_m(A) \rightarrow M_m(A^+)$ is the canonical embedding) 
and thus the second line 
of the last diagram yields the second line of the diagram 
of the proposition, by 
(e) with special $M_2$-action 
again. 

{\bf (g)}		
Summarizing we have also 
fused the first line of the diagram 
of the proposition with $e^{-1}$ to obtain the second line,
namely we have
\begin{eqnarray*}
e^{-1} \cdot s_+ \na_{s_-} 
&=&  e^{-1}  \cdot \iota_A \cdot s_+^+ \na_{s_-^+} = 
\iota \cdot {e^+}^{-1} \cdot s_+^+ \na_{s_-^+}     
= \iota  \cdot r_+ \na_{r_-} 
=   v_+ \na_{v_-}     ,
\end{eqnarray*}
where $v_\pm = r_\pm \circ \iota = s_\pm \otimes \id_m$. 

{\bf (h)}  
In the same vein we can now treat the case $n \neq 1$ 
by fusing with another corner embedding $l: M_m(A) 
\rightarrow M_n(M_m(A))$ to obtain the final
second line of the diagram of the 
proposition with 
$v_\pm = s_\pm \otimes \id_m \otimes \id_n$.   

	{\bf Line 3:} 
	
Our aim is to show that the second line of the diagram of the proposition multiplied by 
$\Delta_{p_-}$ 
is the third line of that diagram. 

	{\bf (m)} 
	Because $l$ is special by assumption, we have also the 
 	corner embedding 
	unitization 
	$l^+: M_m(A)^+   \rightarrow M_n(M_m(A)^+) =:Y$ 
	by Definition \ref{def31} 
	available and 
	can thus build 
	a line  
	$\hat s_+ \na_{\hat s_-}\in GK^G(Y,B)$     
	(see formula \re{eq36})
	like the third line of the diagram  of the 
	proposition by amplification of $s_+ \na_{s_-}$,
	momentarily without the $\square_{p_-} \C$ part, 
	by 
	unitization (see (b)), and extension to range of corner 
	embedding (see (f)). 

	In all that steps, as proven, the $M_2$-action remains 
	special. 
	 
	{\bf (n)} 
	By fusing that so created line 
	$\hat s_+ \na_{\hat s_-}$   with  
	the homomorphism $q: D_- \rightarrow Y$, which is the canonical 
	projection onto the first summand, 
	we get a 
	provisional  third line 
	$w_+ \na_{w_-} := 
	 q \cdot \hat s_+ \na_{\hat s_-} = (\hat s_+ \circ q) \na_{\hat s_- \circ q}$ 
	in the diagram of the 
	proposition.

	{\bf (o)}
  We shall use precise notation
  and thus are going to redefine 
  $P_\pm := p_\pm \oplus \id_\C$. 

	We alter $w_+ \na_{w_-}$ to a final element 
  \be{eq118}
  u_+ \na_{u_-} := (1_{D_-} -  g P_- ) w_+ \na_{w_-}
  =  (w_+ \oplus w_- \circ P_- \circ g  )  
  \na_{w_- \oplus w_+ \circ  P_-  \circ g}   
  \en 
%
  to obtain the 
  third line of the diagram
  of the proposition, 
  where for the last identity we have used 
	 \cite[Lemma 7.2]{gk}
(fusion of a homomorphism with an extended double split
exact sequence), 
\cite[Corollary 9.10]{gk} 
(forming the 
negative of an extended double split exact sequences)
and \cite[Lemma 9.9]{gk} 
(adding two extended split exact sequences). 

	 By (e) and formula \re{eq118}, the $M_2$-space of $u_+ \na_{u_-}$ is special again.
	 
	 The last identity above should not only 
	demonstrate that 
	 a $\Z$-linear span of extended split exact sequences 
	 may be fused to a single one, 
	but how 
	$u_\pm$ are defined.  
	 
{\bf (p)} 
Since $v_\pm$ 
is the restriction of   $u_\pm$,     
we get already
  by \cite[Lemma 7.2]{gk}  
  (fusion with a homomorphism) that
  $ v_+ \na_{v_-} =  i \cdot u_+ \na_{u_-}$. 
  
  By the splitexactness axiom of $GK^G$-theory, see Subsection 
  \ref{subsec21},  we get 
  $$1_{D_-} = \Delta_{P_-} i + g P_-    .$$
  Thus we get the desired identity 
  $$\Delta_{P_-} v_+ \na_{v_-} = \Delta_{P_-} i u_+ \na_{u_-}  
  =  (1_{D_-} -  g P_-) u_+ \na_{u_-} 
  = u_+ \na_{u_-}    ,$$
  where for the last identity we have used 
  the first identity of \re{eq118}.

  {\bf Line 4:} 

  Now the homomorphisms $f_1$ and $f_2$ are the corner 
  embeddings that change the $G$-action on $D$ as explained 
  in Subsection \ref{subsec22}.(c). 
  Again, the fourth line of the diagram is determined by 
  the third line,  
  $f_1$ and 
  result (g) (fusion with the inverse of a corner embedding) and yields $f_1^{-1} u_+ \na_{u_-}  = 
z_+ \na_{z_-}$. 
  
  {\bf Lines 5 and 6:} 
  The fifth line is determined by the fourth line and $f_2$ 
  according to  
  \cite[Lemma 7.2]{gk}
  (fusion of a homomorphism with an
  extended double splitexact sequence) 
  and results into  
	$f_2 z_+ \na_{z_-} = y_+ \na_{y_-}$ 
	with $y_\pm= z_\pm \circ f_1$, 
  and similarly the sixth line by fusing the fifth line with the 
  homomorphism 
  $p_+$ by applying \cite[Lemma 7.2]{gk}, 
  namely $P_+  y_+ \na_{y_-} = t_+ \na_{t_-}$.     


{\bf Revision:} 
To revisit, 
by   
(g)
(fusion of inverse corner embedding with an extended double
split exact sequence)
 we have that $e^{-1} s_+ \na_{s_-}  = 
v_+ \na_{v_-}$ 
and 
$z_+ \na_{z_-} = f_1^{-1} u_+ \na_{u_-}$   
with $z_\pm = u_\pm \otimes \id_{M_2}$. 
By   \cite[Lemma 7.2]{gk}
(fusion of a homomorphism with an extended double split
exact sequence) 
we obtain 
$P_+ f_2 z_+ \na_{z_-} = t_+ \na_{t_-}$ 
for $t_\pm = z_\pm \circ (P_+ f_2)$.  
Finally we recall $\Delta_{p_-} v_+ \na_{v_-} =  u_+ \na_{u_-}$ from (p) above. 
 %

Collecting these identities together 
yields the desired result \re{eq32b}, namely
$$P_+ f_2 f_1^{-1} \Delta_{p_-} e^{-1} \cdot  s_+ \na_{s_-} = t_+ \na_{t_-}   .$$

{\bf Formula for t:} 
The maps $t_\pm$ have the formula 
as notated in the last line of the diagram of the 
proposition    
because, by ignoring 
any $G$-actions, we have 
by the 
formula 
of \cite[Lemma 7.2]{gk} 
that 
$t_\pm = z_\pm \circ (P_+ f_2) = (u_\pm \otimes \id_{M_2})
 \circ f_2 \circ P_+ = 
u_\pm \circ P_+$. 
%
Hence, 
going back to \re{eq118} how $u_\pm$ are defined 
and observing $q \circ (p_\pm \oplus \id_\C) = p_\pm $  ($q$ from (n)) 
we get 
\begin{eqnarray}
t_\pm & = & 
u_\pm \circ P_+= (w_\pm \oplus w_\mp \circ P_- \circ g) \circ P_+
  =  w_\pm \circ P_+ \oplus w_\mp \circ P_-  
		\label{e177}  \\
  &=&  \hat s_\pm \circ q \circ P_+ \oplus   
 \hat s_\mp  \circ q \circ P_- 
= \hat s_\pm \circ p_+ \oplus \hat s_\mp \circ p_-    .  
\nonumber 
\end{eqnarray}

{\bf Very special $GK$-theory:}
For very special $GK$-theory the proof of 
``Line 2" is much simpler. 
Given $s_\pm :(A,\alpha) \rightarrow (X,\gamma)$ in the first line of the diagram of the proposition, and the corner 
embedding $e:(A,\alpha ) \rightarrow 
(M_m \otimes A, \delta \otimes \alpha)$, set  
simply $v_\pm := \id_{M_m} \otimes s_\pm: (M_m \otimes A, \delta \otimes \alpha) \rightarrow 
(M_m \otimes X, \delta \otimes \gamma)$ (for $n=1$). 
The corner embedding $h$ is then very special again with 
$h:(B,\beta) \rightarrow  (M_m \otimes J, \delta \otimes \gamma|_J)$.   
%
%
%
 %
\end{proof}

Propositions \ref{pstand} and \ref{prop02} 
are key to our first main result: 

\begin{corollary}			\label{cor1}

In very special $GK^G$-theory,
for each object $A$, 
every morphism in $GK^G(\C,A)$ 
can be presented
as a level-one morphism, that is, we have
$$GK^G(\C,A) = L^1 GK^G(\C,A)    .$$

\end{corollary}

\begin{proof}

Given a product $z= a_1 a_2 \cdots a_n$ in very special $GK^G$ with generator 
morphisms $a_i \in \Theta$, 
define 
$a_0 := 1_\C$ and turn all generators $a_i$ to level-one morphisms
$\chi(a_i) \in L_1$ according to 
Lemma \ref{lem1}.    
  By 
	Proposition \ref{prop02} we get 
$a_0 a_1 = \chi(a_0) \cdot \chi(a_1) \in L_1$ in $GK^G$. 
If necessary, we turn $a_0 a_1$ to standard form by Proposition
\ref{pstand}, so that we can apply Proposition \ref{prop02} again 
to form the product $a_0 a_1 a_2 = (a_0 a_1) \cdot \chi(a_2) \in L_1$.  
By successively 
repeating this  argument 
(i.e. $a_0 a_1 a_2 a_3 \in L_1$ etc.)   
we finally get $z \in L_1$. If we have given a $\Z$-linear sum of 
$L_1$-elements in $GK^G$
they can be summed up to a level one element in $GK^G$ again by 
\cite[Corollary 9.10]{gk} 
and \cite[Lemma 9.9]{gk},   
yielding the claim. 
\end{proof}


The proof of the last corollary also shows that 
in very special $GK^G$-theory (and by Proposition \ref{prop02} 
partially beyond it) 
we can 
explicitly compute 
the product 
$x \cdot y$ 
for any given $x \in GK^G(\C,A)$ and $y \in GK^G(A,B)$,
what is an important 
tool (for $B=\C$ and in $KK^G$)  for the Dirac and dual-Dirac elements method for the 
proof of the Novikov and Baum-Connes 
conjecture in Kasparov \cite{kasparov1988}, 
Higson and Kasparov \cite{higsonkasparov}, or see 
for instance  
\cite[Theorem 14.1]{guentnerhigsontrout} 
for the concept.  

\section{Classical $K$-theory}

In Corollary \ref{cor1} we have greatly simplified $K$-theory elements in
$GK^G(\C,A)$ to the level-one additive subgroup $L_1 GK^G(\C,A)$, but in the latter
set the complicated and 
intractable relations coming from the 
category $GK^G$  
remain. In this section we clarify this by showing that this 
additive group is just 
ordinary $K$-theory $K^G(A)$. 

To this end, 
in this section we 
entirely switch to very special $GK^G$-theory throughout. 
Recall 
that thus the 
$M_2$-action and their corner 
embeddings $f_1,f_2$ in \re{eq33}
can be omitted, and $s \na_{s} = 0$ in $GK^G$. 



In the next two definitions we define what 
is
meant 
by classical $G$-equivariant $K$-theory $K^G(A)$. 
It must be said that that definition is considerably more restrictive 
with respect to the allowed $G$-actions on projective $A$-modules 
than in 
established $G$-equivariant $K$-theory 
via $KK$-theory  in $C^*$-theory (in the form 
$KK^G(\C,A)$) 
\cite{kasparov1988}, or 
$G$-equivariant $K$-theory 
$K^G$-theory for 
Banach algebras 
for compact groups $G$, 
see for instance Blackadar \cite{blackadar}, 
but for finite discrete groups $G$ both versions of $K^G$ should coincide 
in principle modulo the considered 
class of algebras  
because of the 
validity of the Green-Julg theorem, Lemma \ref{lemma31}. 
Also notice that, meaningfully translated, neither 
Cuntz' universal half-exact $kk$-theory for 
locally convex algebras  
\cite{cuntz}, 
having no $G$-actions at all, nor Ellis'
$G$-equivariant version of $kk$-theory for rings 
\cite{ellis} go beyond 
very special invertible corner embeddings, as exclusively allowed 
from now on, in particular, in the $K^G$-theory 
to be defined next. 

For the next definition recall that we use element descriptions 
$s_+ \na_{s_-} \in L_1$ as in \re{eq33} and their diagram 
descriptions as in \re{eq35} interchangeable.  

\begin{definition}    \label{def41}
{\rm 

Assume $GK^G$-theory is very special. 
Define $S_1$ 
({\em standard form $L_1$}) to be the subset of $L_1$ corresponding 
to these special diagrams, for all $G$-equivariant algebras $A$ 
in $\Lambda^G$:
$$\xymatrix{ 
A      \ar[r] &  M_\infty(M_\infty(A)) \ar[r]      & M_\infty (M_\infty(A)^+)      
\square_{ s_-   
} \C      
\ar@<.5ex>[rr] 
  & &  \C   \ar[ll]^-{s_\pm}      
}   .
$$
A {\em homotopy} in $S_1$ is then an element in $s_+ \na_{s_-} \in S_1(\C,A[0,1])    
 \subseteq L_1(\C,A[0,1])$ 
and its two canonical evaluations $s_+^{(t)} \na_{s_-^{(t)}} \in S_1(\C,A)$ at the endpoints for $t=0,1$ 
are called {\em homotopic} and 
this is notated by $\sim$. 
  (Thereby, identifying $M_\infty(A[0,1])^+ \subseteq M_\infty(A)^+[0,1]$ at first. 
Confer also \cite[Lemma 8.10]{aspects} for similar 
homotopies.) 
An {\em addition} $\oplus$ on $S_1(\C,A) \subseteq L_1(\C,A)$ is declared 
by taking direct sums of elements in $S_1$, 
that is $s_+ \na_{s_-} \oplus t_+ \na_{t_-} :=
  (s_+ \oplus t_+ )\na_{s_-  \oplus t_-}$, 
confer also 
the similar   \cite[Lemma 9.9]{gk}. 
   
}

\end{definition}

\begin{definition}[Classical $K$-theory]		\label{def42} 
{\rm

For each object $A$, 
define {\em classical equivariant $K$-theory} $K^G(A)$ 
as the set-theoretical quotient of $S_1(\C,A)$ divided by 
the equivalence relation $\equiv$ defined by 
(for all $x,y \in S_1(\C,A)$)  
$$ x \equiv y 
\qquad  \Leftrightarrow  \qquad \exists u, v \in \hom  
\big (\C,M_\infty (M_\infty(A)^+)    \big )
: \quad  
  x \;\oplus  \; u \na_u 
\sim  y \; \oplus \; v \na_v$$
(equivalence if and only if adding 
trivial elements is homotopical). 
Addition $\oplus$ on $K^G(A)$ is declared to be the one induced by 
taking direct sums in $S_1(\C,A)$ 
as described in Definition \ref{def41}.  
%
}
\end{definition}

It is easy to see by standard arguments that $\big (K^G(A),\oplus \big )$ is an 
abelian group, namely by the proof of \cite[Corollary 9.10]{gk} 
(rotation homotopy) one has 
 $-s_+ \na_{s_-}= s_- \na_{s_+}$. 

Since we can always form the standard form of elements in $K^G(A)$ 
by the method of Proposition \ref{pstand}, $K(A)$ has the formal 
shape 
as Cuntz' $kk(\C,A)$ for locally convex  algebras $A$ 
by \cite[Theorem 7.4]{cuntz} and Phillips' 
$K$-theory  $R K_0(A)$ for  Fréchet algebras $A$ \cite{phillips}, 
which generalizes 
$K$-theory for Banach algebras \cite{blackadar}.

The next main result is similar to 
Corollary \ref{cor1} with a
similar structural proof. But now we are going to check 
how the 
equivalences in $K$-theory within 
$GK^G$-theory play down to
the 
equivalences in classical $K$-theory, thus obtaining 
the improved result.

\begin{theorem}    \label{thm42}

In very special $GK^G$-theory, 
for each object $A$, the group $GK^G(\C,A)$ is classical $K$-theory.
In other words, we have 
an abelian group isomorphism 
$$GK^G(\C,A) 
\cong  K^G(A)  . $$

\end{theorem}

\begin{proof}
 Our goal is to define a well-defined map 
 $\Phi: L(\C , A) \rightarrow S_1 (\C,A)$ 
 which canonically descends 
 to $\Psi: GK^G(\C,A)   \rightarrow K^G(A)$. 

{\bf Definition of $\Phi$.} 
 Let a representant $\alpha \in L(\C,A)$ of $GK^G(\C,A)$ be given. 
 Expand $\alpha$ by 
 multiplying all brackets appearing 
 in $\alpha$ according to the distribution and associativity laws out (in any order), and call this new element $\beta \in L(\C,A)$.
 
 Since the order of how this is done is irrelevant at the end,
 this operation is well-defined, so a {\em  map} $\beta=\beta(\alpha)$, and because of the associativity and distributions laws included in 
 $R$ we get $\alpha  \equiv_R \beta$. 
 This element is of the form
 $\beta = \sum_{k=1}^n  \prod_{i=1}^{n_k} a_{k,i}$, where
 $a_{k,i} \in \Theta$ are generators. 

  To proceed with the set $N \subseteq L$ comprised of all such expanded elements $\beta$ of that form, we recursively define 
  a {\em map} $F:N \rightarrow S_1$ 
  as 
  $$F \Big (\sum_{k=1}^n  \prod_{i}^{n_k} a_{k,i} \Big )
   := \bigoplus_{k=1}^n   \Big (...((( \chi(\id_\C) 
   \odot  a_{k,1}) \odot a_{k,2} ) \odot  a_{k,3} )  ...  \odot a_{k,{n_k}}  \Big )  , $$
   where the ``multiplication'' map $\odot$ is defined as
   \be{eq19}
   Z:S_1  \times 
   \Theta \rightarrow S_1: 
   (x,a) \mapsto x \odot a  := S \big (P \big ( x , \chi(a)  \big ) 
   \big ) 	,
   \en
   provided the product is valid in the sense that $x$ and $a$ is composable to $xa$ in $L$ because range and source fit together. 
   The function  $\Phi$ is then defined by $\Phi(\alpha):= F(\beta(\alpha))$. 
Note that $\Phi$ is additive as both $F$ and $\beta$ are. 
   
   Here, 
	$S: L_1 \rightarrow S_1$ is a function, where $S(x)$  means the level-one standard form 
   $t_+ \na_{t_-} \in S_1$   
   of $x \in L_1$ 
   that comes out 
   when applying the {\em formulas} 
	$t_\pm := \big ( s_- \oplus 
q \big) \cdot 
\ad (U_{\pi/2})$  
of 
   Proposition  
   \ref{pstand}  
   to $ s_+ \na_{s_-} := x$. 
    
    Further, $\chi: \theta \rightarrow L_1$ is the function, where $\chi(a) \in L_1$ means the 
    level-one representations for the generator 
	$a \in \Theta$ according to the {\em formulas} of 
Lemma \ref{lem1}. 
     
     Finally, 
$P: S_1 \times L_1 \rightarrow L_1$ 
is a function 
(the domain of it restricted to the composable elements), 
where 
	$P \big ( x , y )$ stands 
     for the level-one element $t_+ \na_{t_-} \in L_1$ that comes out when applying
     the 
	{\em formulas} 
	$t_\pm:=
  \hat s_\pm   \circ p_+ 	\oplus \hat s_\mp \circ p_- $
of 
	Proposition \ref{prop02}  to the elements $p_+ \na_{p_-} := x \in 
S_1$ 
     and  $s_+ \na{s_-} := y \in L_1$   
     ($p_\pm$ and $s_\pm$ obviously defined). 
      
      Note that we do not apply 
		Lemma \ref{lem1} nor 
		Propositions \ref{pstand} and \ref{prop02} directly,
      but only use their formulas, 
      which may also be 
      used in $L$.
   Summarizing, roughly speaking we simply define 
   $\Phi$ by successive application of 
		Propositions 
   \ref{pstand} and   
   \ref{prop02}. 


 {\bf Equivalences to be checked.} 
We need now 
check if 
the 
equivalence relation $R$ 
defining $GK^G$, see Subsection \ref{subsec21},    
passes through $\Phi$ to the equivalence
relation $\equiv$		
defining $K^G$, see Definition \ref{def42}.  
Hence we 
claim that, for all $x,y \in S_1 (\C,A)$ and
$a \in \Theta$ (and provided composability in what follows), 
\be{e5}
x \equiv y  \quad \Rightarrow \quad  x \odot a \equiv y \odot a   , 
\en 
\be{e3}
(x \oplus y) \odot a \equiv  (x \odot a) \oplus  (y \odot a)    .
\en

Moreover, 
we need to verify that 
\be{e18}
x \equiv x  \odot 1_A   , 
\en
\be{e7}
((x \odot \varphi) \odot \psi) \equiv x \odot (\varphi \psi) 
\en
for all composable  homomorphisms $\varphi$ and $\psi$, 
and if $a:  A \rightarrow B[0,1] $ is a homotopy that 
\be{e10}
x \odot a_0 \equiv y \odot a_1   .
\en 
 The other relations come from split exactness and corner 
 embeddings and we need to check that
\be{e1}
x \equiv ((x \odot \Delta_s ) \odot i) \oplus ((x \odot f) \odot s)  ,
\en 
\be{e3b}
x \equiv ((x \odot  i) \odot \Delta_s)   ,
\en 
\be{e8}
x \equiv ((x \odot e^{-1} ) \odot e), \qquad 
x \equiv  
((x \odot e ) \odot e^{-1})    .
\en

To revisit, the  
above assertions and relations tell us that the structure 
of $K^G$-theory and the defining relations of $GK^G$-theory, 
respectively,  
pass through $Z$. 

To keep the proof shorter, let us take these claims for granted and outsource their proof to later and finalize the theorem. 

{\bf The equivalence relation passes through $\Phi$.}  
Let us observe how the 
most complicated relation of $R$,  
$\Delta_s i + f s \equiv_R 1_M$,    
see Subsection \ref{subsec21}.(f),  passes through $\Phi$. 
   If we consider two $\alpha$s in $L$, one version where one instance of 
$1_M$  appears somewhere, and another version where only this one 
	$1_M$ is replaced by
   $(\Delta_s i + f s)$ and observe how these $\alpha$s pass through $\Phi$
   and thus through $F$, then we will realize that we need to check 
   that 
   $$F(a_1 ... a_n \Delta_s i a_{n+1} ... a_m + a_1 ... a_n f s  a_{n+1} ... a_m ) \equiv  F(a_1 ... a_n 
1_M a_{n+1} ... a_m)     ,$$
where $a_i \in \Theta$ are generators. 

We prove this by induction. 
 The induction start, 
 if $m=n$, is clear by \re{e1} 
 and \re{e18} for $x:= F(a_1 ... a_n)$ and 
 the definition of $F$. 

Induction step $m \Rightarrow m+1$:
By 
definition of $F$ we get 
$$F(a_1 ... a_n \Delta_s i a_{n+1} ... a_m a_{m+1}+ a_1 ... a_n f s  a_{n+1} ... a_m a_{m+1})$$
$$= F(a_1 ... a_n \Delta_s i a_{n+1} ... a_m) \odot a_{m+1}  \oplus  F(a_1 ... a_n f s  a_{n+1} ... a_m ) \odot a_{m+1}$$
$$\equiv F(a_1 ... a_n 
1_M a_{n+1} ... a_m) \odot a_{m+1}       ,$$

where for the last step we have used \re{e3}, \re{e5}  
and the induction hypothesis.  

We have checked that both $\Phi(\alpha)$s are equivalent with respect to $\equiv$. Similarly we do so with all other relations 
defining $\equiv_R$. Note that associativity and distribution laws  
are already handled from the transition from $\alpha$ to $\beta$, 
because different $\alpha$s by these laws yield the same $\beta$.  
Hence $\Phi$ descends to the desired map $\Psi$.

{\bf 
Injectivity of $\Psi$.}  

If $\Phi(\alpha) \equiv 0$ is zero in $K^G$-theory, then certainly also the representant $\alpha$ is zero in $GK^G$, 
because the relations $\equiv$ of $K^G$-theory hold in $GK^G$ theory as well, as well as the associativity and distribution laws, 
and  by 
	Lemma \ref{lem1}
 and 
	Propositions  \ref{pstand} and  \ref{prop02} the product
 $Z$ descends just to the ordinary product in $GK^G$-theory, 
 so it does not matter if we multiply $\alpha$ out and simplify it 
 or not. 
 

{\bf Surjectivity of $\Psi$.}   
The proof is outsourced to Lemma \ref{lemma51}.    
\end{proof}

\begin{lemma}			

The yet 
unproven claims \re{e5}-\re{e8}
of the last theorem are correct.

\end{lemma}

\begin{proof}

{\bf Assertions \re{e5} and \re{e3}. } 
The 
formulas of  
Proposition 
\ref{pstand} 
show that the standard form $S$ 
is computed by $S(s_+ \na_{s_-}) 
= t_+ \na_{t_-}$  with 
 $t_\pm := (s_\pm  \oplus q) \cdot \ad(U_{\pi/2})$.  
%
Thus it is easy to see that 
the standard form $S$   
preserves 
direct sums, 
homotopy and trivial elements $p \na_{p}$, that is, a homotopy $s_+ \na_{s_-}$ yields 
a homotopy $t_+ \na_{t_-}$ such that 
its 
evaluation at time $\lambda \in [0,1]$ is 
$t_\pm^{(\lambda)} = (s_\pm^{(\lambda)}  \oplus q^{(\lambda)}) 
\cdot  \ad(U_{\pi/2}^{(\lambda)}) = s_\pm^{(\lambda)}$, 
and similarly so for direct sums and trivial elements.

  Similarly 
  it is easy to see  that the fusion formula  $t_+ \na_{t_-} 
	= (\hat s_+ \circ p_+  \oplus \hat s_- \circ p_- ) \na_{\hat s_- \circ 
p_+   \oplus \hat s_+ \circ p_- } \in L_1$
  of 
	Proposition \ref{prop02}  preserves direct sums,  
homotopy and trivial elements in the variable $p_+ \na_{p_-} \in L_1$ for each (applicable) $s_+ \na_{s_-} \in L_1$,  
  whence the function $Z$ stated in \re{eq19} 
  preserves direct sums, 
	homotopy and trivial elements  in the variable $x$.  
  Hence claims \re{e5} and \re{e3} are valid, 
since  each, $S$ and $P$, 
	respect the equivalence $\equiv$.   

{\bf Standard form.}
Since  
the formula of  
Proposition \ref{pstand}  
shows 
that the standard form is performed by adding on a 
trivial cycle 
and then performing a homotopy, 
by Definition \ref{def42}, 
if $x \in S_1$ then its standard form is  
equivalent in $K^G$ 
to itself, that is,    
\be{eq129}
S(x) \equiv x   .
\en


{\bf A simpler formula for $t$.} 
If $z= p_+ \na_{p_-}$ is such that $p_-$ maps into the algebra
$M_m(\C)$ of units of $M_m(M_n(\C)^+))$, for example 
if $z = S(x)$ is coming from the standard from 
formula, 
then the last line $t_+ \na_{t_-}$ of 
Proposition \ref{prop02} reduces to the simple 
form, equivalent in $S_1$ with respect to $\equiv$, 
\be{eq200} 
\xymatrix{  
 B  \ar[r]  
  & M_{n}(M_m(J)) \ar[r] &  M_{n}(M_m(X)^+) \square_{t_-}   
 \C   
\ar@<.5ex>[rrr]  
 &   &&
\C       
\ar[lll]^-{ t_\pm:=
  \hat s_\pm   \circ p_+  }	  
}
\en  
because by 
the formula of the last line of the diagram of Proposition \ref{prop02} 
then 
the summand $\hat s_\pm \circ p_-$ is a trivial element and can be subtracted in $K^G$. 

Because of the 
verified relations \re{e5}-\re{e3} 
and $S$ preserves $\equiv$ as shown,  
it is clear that the relations \re{e18}-\re{e8}   
are 
verified 
too whenever we show them only for $x$ of the form
$x= S(z)$, and thus, making that assumption, can use the simple formula 
\re{eq200} instead of the more complicated formula 
stated in the last line of the diagram of Proposition 
\ref{prop02},   
and thus in formula $Z$, \re{eq19}, if we need only equivalence 
up to $\equiv$, because both $S$ and $P$ respect this 
equivalence by the beginning of this proof.   
We shall assume this for the rest of the proof without much saying. 

{\bf Identity \re{e1}.}
{\bf (a)}  
Let us check \re{e1}. 
To this end, we need at first form 
$z \odot \Delta_s$ 
for $\Delta_s$ of 
the split exact sequence \re{eq32}.  
According to formula \re{eq19}, 
we must take 
$p_+ \na_{p_-}:=  
z \in S_1$,   
paint 
its diagram 
into to most right row of the diagram of Proposition 
\ref{prop02},  
then take the diagram \re{e12}, which represents $\chi(\Delta_s)$,  
and write it into the first line of the diagram of  Proposition 
\ref{prop02}. 
Then we apply this proposition, and read off the resulting diagram,
which represents $P(x, \chi(\Delta_s))$, 
in the 
last line of the diagram
of the 
proposition.  
Now we use the simpler version \re{eq200} 
as explained above by assuming 
$z = S(x)$. 
What comes out is
\be{eq23}
\xymatrix{
  J  \ar[r] 
  & M_n(M_m(J)) \ar[r] &  M_n(M_m(M)^+) \square_{t_-}   
 \C   
\ar@<.5ex>[rrr]
 &&&   
\C  
\ar[lll]^-{ t_+ = \widehat {\id} \circ p_+, \, 
 t_- = \widehat {f  s} \circ p_+ }
}   .
\en

Finally, according to formula \re{eq19} we need to form the standard form $u_+ \na_{u_-}$
of $t_+ \na_{t_-}$ 
by the formulas of  Proposition 
\ref{pstand}. 

{\bf (b)}
The next step is to produce 
$(z  \odot \Delta_s) \odot i$,
that is, $(u_+ \na_{u_-}) \odot i$.

Then the diagram 
corresponding to $u_+ \na_{u_-}$,  
that is 
\be{eq21}
\xymatrix{
  J  \ar[r] 
  & M_{2 n}(M_m(J)) \ar[r] &  M_{2 n}(M_m(J)^+) \square_{u_-}   
 \C   
\ar@<.5ex>[rrr]
 &&&   
\C  
\ar[lll]^-{ u_\pm }    	
}      ,
\en
 has to be filled into 
the 
last row of the diagram of 
Proposition \ref{prop02}, 
and the diagram 
corresponding to the homomorphism $i$ 
by formula \re{eq10},
that is
$$\xymatrix{ 
M     \ar[r]^\id &  M \ar[r]     & M    
\square_{0 } J     \ar@<.5ex>[rr]
 &&  J   \ar[ll]^{i,\; 0} 	,
     }
$$
into the first line of the diagram of that proposition. 
The 
outcoming result, to be read off in the last line of the diagram 
of the proposition (actually \re{eq200}), 
is 
\be{eq27b}
\xymatrix{
  M  \ar[r] 
  & M_{2n}(M_m(M)) \ar[r] &  M_{2n}(M_m(M)^+) \square_{v_-}   
 \C   
\ar@<.5ex>[rrr]
 &&&   
\C  
\ar[lll]^-{ v_+ = \hat {i} \circ u_+, \, 
 v_- = \hat {0} \circ u_+ } 
}  .
\en

Because $i: J \rightarrow M$ is the 
embedding, and the formulas of 
the matrix entries of  $u_+(\lambda)$ and $u_-(\lambda)$ from 
Proposition \ref{pstand} 
are  
algebraic expressions in $t_\pm(\lambda) \in X^+$ and $1_{X^+} \in X^+$ in the algebra $X^+$ which 
happen to 
land in 
$J^+ \subseteq X^+$, it is clear that the matrix entries of $\hat i ( u_+(\lambda))$ 
and $\hat 0 (u_-(\lambda))$ are just $u_+(\lambda)$ and $u_-(\lambda)$ again 
(the latter because $u_-(\lambda)$ involves only $1_{X^+}$
and $0$ besides the scalar $\lambda$),
but now interpreted 
as elements in  
$M_{2n}(M_m(X^+))$. 
Consequently, $v_\pm = u_\pm$ in the last diagram 
as formulas. 

Thus \re{eq27b}
is the same as at first replacing $J$ by $M$ 
everywhere in the diagram
\re{eq23}, 
notated 
\be{eq28}
\xymatrix{
  M  \ar[r] 
  & M_n(M_m(M)) \ar[r] &  M_n(M_m(M)^+) \square_{T_-}   
 \C   
\ar@<.5ex>[rrr]
 &&&   
\C  
\ar[lll]^-{ T_+ = \widehat {\id} \circ p_+, \, 
 T_- = \widehat {f  s} \circ p_+ }
}   ,
\en
and then forming its standard 
form, 
notated again by $v_+ \na_{v_-}$. 

{\bf  (c)} 
Now, according to \re{e1} we produce $((z \odot f) \odot s)$. 
First of all, by the above discussion and the similarity of all involved formulas, we see 
that $((z \odot f) \odot s)$ is the same 
as forming the standard form of $z \odot (f  s)$, because the standard form 
is formed twice in $((z \odot f) \odot s)$. 
The complicatedness of changing the ideal $J$ to $M$ as above does here even not appear. 

In particular we also have checked the relation \re{e7}. 

{\bf (d)} 
Now up to equivalence $\equiv$, $z \odot (f s)$ corresponds to 
the standard form of  
the  diagram 
(by 
\re{eq200} and assuming $z=S(x)$ is in standard form)  
\be{eq25}
\xymatrix{
  M  \ar[r] 
  & M_{2 n}(M_m(M)) \ar[r] &  M_{2 n} (M_m(M)^+) \square_{w_-}   
 \C   
\ar@<.5ex>[rrr]
 &&&   
\C  
\ar[lll]^-{ w_+ = \hat {f s} \circ p_+, \, 
 w_- = \hat {0} \circ p_+ }
}   .
\en

Thus by \re{eq129} we get 
\begin{eqnarray}		\label{eq26}
&& ((z \odot \Delta_s ) \odot i) \oplus ((z \odot f) \odot s)   \\
\nonumber 
&\equiv&  S( T_+ \na_{T_-}) \oplus S(w_+ \na_{w_-})    
\equiv  T_+ \na_{T_-} \oplus w_+ \na_{w_-}   \\
&\equiv&
T_+ \na_{T_-} \oplus T_- \na_{w_-} 
\equiv
  (T_+ \oplus T_-)\na_{T_- \oplus w_-}
  \equiv
  (T_- \oplus T_+)\na_{T_- \oplus w_-}   \nonumber \\
&\equiv&
  T_+\na_{w_-}
  \equiv 
  ( \hat \id \circ p_+) \na_{\hat {0} \circ p_+}
    \equiv 
  z 
 \odot \id    \equiv z  , 		
						\nonumber
\end{eqnarray}
where 
we 
trivially observed $T_+ = w_-$, rotated $T_+ \oplus T_-$ 
to $T_- \oplus T_+$ 
by an 
ordinary  
rotation homotopy in $K^G$, 
which is possible because both $T_+$ and $T_-$ 
map into the same space 
with the same $G$-action as always in very special $GK^G$-theory,
and used the addition formula of elements in $K^G$ and $T_- \na_{T_-} = 0$. 

 The last identity 
of \re{eq26} 
is because of \re{e18}, to be checked next.  
This proves \re{e1}. 
Relation \re{e3b} is similarly deduced and easier. 

{\bf Relation \re{e18}.}
  Let $\varphi : A \rightarrow B$ be a homomorphism 
  and 
	$z \in S_1(\C,A)$   
(assuming $z=S(x)$ is in standard form).  
At first compute 
$z \odot \varphi$  
by
evaluating the formula $t_\pm$ of  Proposition 
\ref{prop02} 
for $s_+ \na_{s_-} = \varphi \na_{0}$
according to diagram \re{eq10}
and 
$z$ denoted by $p_+ \na_{p_-}$. 

Then, by adding and subtracting zero elements of the form
$h \na_h$ 
and using \re{eq129}, 
we obtain (and 
applying \re{eq200}) 
\begin{eqnarray}   
		\label{eq66}
z \odot \varphi &=&  S \big (  (\hat \varphi \circ p_+)  
\na_{  \hat 0 \circ p_+ }   \big )  
\equiv 
(\hat \varphi \circ p_+ \oplus  \hat \varphi \circ p_-) 
\na_{  \hat 0 \circ p_+   \oplus  \hat \varphi \circ p_-}    \\
&\equiv&   
(\hat \varphi \circ p_+   \oplus  
\hat 0 \circ p_+) 
\na_{ \hat 0 \circ p_+   \oplus 
\hat \varphi \circ p_- }
\equiv (\hat \varphi \circ p_+ ) 
\na_{ 
\hat \varphi \circ p_- } 	, 
    \nonumber
\end{eqnarray} 
where we have also used that 
$\hat \varphi \circ p_- = \hat 0 \circ p_- 
= \hat 0 \circ p_+$, since $p_-$ maps only into the matrix $M_n(\C) 
\subseteq M_n( M_m(A)^+)$ of units 
by assumption of the standard form of $p_+ \na_{p_-}$, and $p_+(1_\C) - p_-(1_\C)$ 
is in the ideal $M_n(M_m(A))$ 
by Subsection \ref{subsec22}.(c), which does not pass through $\hat 0$.
%
For $\varphi = \id$ this verifies \re{e18}.

{\bf Identity \re{e8}.}  
{\bf (a)} 
Now let us focus onto the first identity of \re{e8}.
Let $e : A \rightarrow M_k(A)$  be a corner embedding. 
Again, we go back to the diagram of   Proposition 
\ref{prop02},
assume that its last row 
represents 
a given $z
\in 
S_1 (\C,M_k(A))$ 
(that means, each $A$ in that row is replaced by $M_k(A)$), 
and assume that the first line is   
the diagram  \re{e11} 
for $J$ replaced by $M_k(A)$.  
(Observe that $X$ of the diagram of 
Proposition \ref{prop02} 
becomes $M_k(A)$.) 
Forming $z \odot e^{-1}$ then only yields 
the standard form of 
the diagram  (by diagram \re{eq200}) 
\be{eq27}
\xymatrix{
  A  \ar[r] 
  & M_n(M_m(M_k(A))) \ar[r] &  M_n(M_m(M_k(A))^+) \square_{p_-}   
 \C   
\ar@<.5ex>[rrr]
 &&&   
\C  
\ar[lll]^-{  \widehat \id \circ p_+, \, \hat 0 \circ p_+ } 
}   .
\en

But by computation  \re{eq66} this is equivalent to
$\widehat \id \circ p_+ \na_{\widehat \id \circ 
p_-}$, which is
$p_+ \na_{p_-}$, because $\widehat \id = \id$ here. 
As compared to $z$, only the final algebra has changed to $A$.

{\bf (b)} 
Next we need to compute $(z \odot e^{-1}) \odot e$, 
and to this end 
directly apply the computation \re{eq66}, which yields
the result 
$\hat e   \circ p_+ \na_{
\hat e   \circ p_-}$.  
In diagram form it is
%
%
$$\xymatrix{
  M_k(A)  \ar[r] 
  & M_n(M_m(M_{k^2}(A))) \ar[r] &  M_n(M_m(M_{k^2}(A))^+) \square   
 \C   
\ar@<.5ex>[rr]		
 &&		
\C  
\ar[ll]^-{ 
 \hat e \circ p_\pm}			
}   .
$$

As compared to 
$z$, we are left 
with an additional corner embedding in the middle
of the diagram. 
But this is the same as if we had added the zero element
$0 \na_0$, where  
$0_\pm: \C \rightarrow M_n(M_m(  
M_{k^2-k}   (A))^+ )$, 
to $p_+ \na_{p_-}$,  
and thus we can 
remove it by subtraction in $K^G$,
and end up with the standard form of $z$.

This proves the first relation of \re{e8}, and the 
second relation 
is completely analogous. 

{\bf Identity \re{e10}.}  
Identity \re{e10} is also easy to check 
by formula \re{eq19}. If we enter a homotopy 
homomorphism as in \re{eq10} 
into the first line of the diagram of Proposition 
\ref{prop02}, a homotopy in $L_1$, see \cite[Lemma 8.10]{aspects},  will come out in the last 
of that diagram, which becomes a homotopy in $S_1$ 
when forming its standard form as already remarked at the beginning 
of this proof.  
\end{proof}

  





\begin{lemma}		\label{lemma51}

The map 
$\Psi$ of the proof of Theorem \ref{thm42} 
is surjective 
and $\Psi^{-1}(z)=z$. 

In particular, the  
isomorphism $\Psi^{-1}: K^G(A) \rightarrow GK^G(\C,A)$ is trivially induced by the 
identical  
embedding $S_1 \rightarrow L$. 

\end{lemma}

\begin{proof}

Let us an element $p_+ \na_{p_-}$ in $S_1(\C,A)$ as drawn in 
diagrammatic 
form in the most right column of the diagram
of 		Proposition 
\ref{prop02} be given. 
By reasons explained in 
``a simpler formula for t" in the proof 
of Lemma \ref{lemma31}  we suppose that 
$p_+ \na_{p_-} = S(z)$ is in standard form.
As in \re{eq33}, 
more detailed notated it is $p_+ \na_{p_-} = P_+  \Delta_{P_-}
e^{-1}$ for 
$$P_\pm := p_\pm \oplus \id_\C : 
\C \longrightarrow 
D_\pm :=  M_n (M_m(A)^+) \square_{p_\pm}  \C   .$$
  

We propose it is the image $\Phi(P_+  \Delta_{P_-}
e^{-1})$ up to equivalence $\equiv$, 
verifying that 
$\Psi$ is surjective, thus bijective,
and 
the last claims of the lemma.   
To prove this, we compute it step by step and 
begin with 
$$\Phi(P_+ ) = F(P_+ )= \chi(1_\C) \odot P_+ 
= S 
\big (P 
\big (\chi(1_\C) ,  
\chi(P_+  \big ) \big ) = S \big (\chi(P_+ ) \big ).$$ 
Here, the used 
identity $P 
(\chi(1_\C) ,  z ) = z$ is easy to see,  
so let us skip this step.

Now $\chi(P_+ )$ 
corresponds to the 
extended double split exact sequence 
\be{eq106}
\xymatrix{ 
 D \ar[r]     
&   D   \ar[r]   & 
D      
\square_0 
\C  
		\ar@<.5ex>[rr]   
 & &      \C   
\ar[ll]^{  P_+      
, \, 0 }       
}   .
\en 

Because of the $0$ split it has already sufficient standard form 
as explained in the paragraph before \re{eq200}, 
and thus we can omit computing its $S$ by \re{eq129} and use \re{eq200} next.   
As already discussed in 
``Identity \re{e1}'' 
of 
Lemma \ref{lemma31}, 
see result 
\re{eq23},  fusing it with $\Delta_{P_-}$ 
(that is, 
computing  $\Phi(P_+ \Delta_{P_-})   
= 
\Phi(P_+) \odot  \Delta_{P_-}$)  
is 
the standard form of 
a 
diagram of the form
\be{eq110}
\xymatrix{ 
 N \ar[r]^{\id}		
&   N   \ar[r]   & 
D      
\square_{t_-} 
\C   
	\ar@<.5ex>[rrrr]
& &  & &    \C   
\ar[llll]^-{ t_+:=\widehat \id \circ  P_+   ,  
     \; t_- :=\widehat{g P_-} \circ  P_+     }       
}       ,    
\en 
where  $N:= M_n (M_m(A))$   
and $g: D \rightarrow \C$ is the canonical projection of the split exact sequence of the last column of the diagram of 
Proposition 
\ref{prop02}. 
 
 Because we have constant three $D$s in \re{eq106}, $\widehat \id = \id$
 and $\widehat  {g P_-} = {g P_-}$. 
 
 Consequently, $t_+ = P_+$ and $t_- = P_- \circ g \circ P_+
 = P_-$, and thus 
diagram \re{eq110}  has already sufficient standard form.

As already observed in the proof part ``Identity \re{e8}"
of 
   	Lemma \ref{lemma31},
fusing it with the inverse corner embedding $e^{-1}$  
(that is, 
computing  $\Phi(P_+ \Delta_{P_-} e^{-1} )  
= 
\Phi(P_+  \Delta_{P_-} ) \odot   e^{-1} $)  
yields 
%
%
\be{eq115}
\xymatrix{ 
 A \ar[r]^{e }   
&   N   
\ar[r]  		&
D		
\square_{P_-}			
\C   
\ar@<.5ex>[rr] 
 & &      \C   
\ar[ll]^-{ P_+  , \;    P_- }  
}   .
\en 

Here $D$ is not quite 
desired.  To get 
precisely the 
result with $Y:=M_n (M_m(A)^+)$ instead  of $D$, we 
note that we have a 
canonical 
homomorphism 
$q : D \rightarrow Y$ projecting into the first summand of $D$,   
and replace the 
extended split exact sequence associated to 
$P_+ \Delta_{P_-} e^{-1}$  in the above computation  
by   
the extended split exact sequence associated to 
$$   
\xymatrix{ 
    A \ar[r]^e   & N \ar[r]   
  &  Y    
\ar@<.5ex>[r]^h    \ar@<-.5ex>[l]_{\Delta_{p_-}}   
 &  \C   \ar[l]^{ p_\pm }       
     }	, 
    $$
which is 
equivalent in $GK^G$ to the former one 
(indeed connect the diagrams of 
$P_+ \Delta_{P_-} e^{-1}$ and $p_+ \Delta_{p_-} e^{-1}$  
with canonical arrows and the arrow $q$ in the middle 
and verify with \cite[Lemma 4.3.(iii)]{gk}) 
and thus gives the same 
result 
\re{eq115} in $K^G$ as before since we already have proven 
that $\Phi$ is well-defined, 
but this is now evidently 
the at the very beginning given $p_+ \na_{p_-}$ by the so 
modified  diagram 
\re{eq115}, that is, $D$ replaced by $Y$.   
%
\end{proof}

$GK^G(\C,-)$ is 
a functor $\Lambda^G \rightarrow GK^G$ by assigning $A$ to $GK^G(\C,A)$ 
and $f:A \rightarrow B$ to the map 
$f_*(z) := z \cdot f$ (product in $GK^G$) 
for $z \in GK^G(\C,A)$.   

$K^G$ 
may be regarded as a category of abelian groups and abelian group 
homomorphisms, 
and as a functor as well, as usual: 

\begin{definition}		\label{def51} 
{\rm 

$K^G$ 
is a functor 
$\Lambda^G \rightarrow K^G$ assigning $A$ to $K^G(A)$ and 
$f:A \rightarrow B$
to $K^G(f)(s_+ \na_{s_-}) := (\hat f  \circ  s_+ ) \na_{(\hat f \circ s_- )}$
for all $s_+ \na_{s_-} \in S_1(\C,A)$, see \re{eq36} for $\hat f$.     

}
\end{definition}

\begin{lemma}

The isomorphism of 
Theorem  \ref{thm42} 
respects functoriality in the variable $A$. 

\end{lemma}

\begin{proof}

Recall that the isomorphism is performed by a map $\Phi$ on the $L$-level, 
and by 
Theorem 
\ref{thm42} and 
Lemma \ref{lemma51}  we may assume $z \in S_1$ 
and $\Phi(z)=z$.
Then for a homomorphism $f:A \rightarrow B$ we get 
$\Phi(z \cdot f) = \Phi(z) \odot f = z \odot f = K^G(f)(z)$, 
where the last identity is by identity \re{eq66} 
and 
Definition \ref{def51}. 
%
%
\end{proof}

\bibliographystyle{plain}
\bibliography{references}

\end{document}